\documentclass{article}
\usepackage{amsmath,amsthm,thm-restate,thmtools,amssymb}
\usepackage{graphicx}
\usepackage{booktabs}
\usepackage{floatrow}
\usepackage{subfig}
\usepackage{blkarray}
\usepackage{authblk}
\usepackage[normalem]{ulem}
\usepackage{libertine} 
\usepackage{multirow,bigdelim}   
\usepackage{nicefrac}
\usepackage{color, soul}
\newtheorem{thm}{Theorem}
\newtheorem{lemma}[thm]{Lemma}
\newtheorem{cor}[thm]{Corollary}
\newtheorem{exmp}[thm]{Example}

\newcommand{\er}{Erd\H{o}s-R\'{e}nyi }

\usepackage{algorithm,algorithmic}

\def\mdots{\vbox{\baselineskip=3pt \lineskiplimit=0pt 
\kern2pt \hbox{.}\hbox{.}\hbox{.}}}
\newcommand{\myvdots}{\phantom{1}\llap{\mdots\hspace{0.25ex}}}
\newcommand{\bigdots}{\multirow{2}{*}{\vdots}}
\newcommand{\ob}[1]{\begin{array}{|c|}\hline ~\makebox[0pt]{$#1$}~\end{array}}
\newcommand{\mi}[1]{\begin{array}{|c|}~\makebox[0pt]{$#1$}~\end{array}}
\newcommand{\un}[1]{\begin{array}{|c|}~\makebox[0pt]{$#1$}~\\\hline\end{array}}
\newcommand{\LiKl}[1]{\delimitershortfall=-4pt\ldelim({#1}{0pt}}
\newcommand{\ReKl}[1]{\delimitershortfall=-4pt\rdelim){#1}{0pt}}

\title{A unifying framework for fast randomization of ecological networks with fixed (node) degrees}
\author[1]{C. J. Carstens}
\author[2]{A. Berger}
\author[3]{G. Strona}
\affil[1]{Korteweg-de Vries Institute for Mathematics, University of Amsterdam, Amsterdam, The Netherlands}
\affil[2]{Institute of Computer Science, Martin Luther University Halle-Wittenberg, Halle (Saale), Germany}
\affil[3]{European Commission, Joint Research Centre, Directorate D - Sustainable Resources - Bio-Economy Unit, Ispra, Italy}
\begin{document}
\maketitle

\begin{abstract}
The switching model is a Markov chain approach to sample graphs with fixed degree sequence uniformly at random. The recently invented Curveball algorithm \cite{Strona2014} for bipartite graphs applies several switches simultaneously (`trades'). Here, we introduce Curveball algorithms for simple (un)directed graphs which use single or simultaneous trades. We show experimentally that these algorithms converge magnitudes faster than the corresponding switching models. 
\smallskip

\end{abstract}

\section{Introduction}
The uniform sampling of bipartite, directed or undirected graphs (without self-loops and multiple edges) with fixed degree sequence has many applications in network science \cite{MolloyReed1995,NewmanStrogatzWatts2001,Artzy-Randrup2005,CarstensCompleNet2014,gotelli2009}. In this paper we focus on Markov chain approaches to this problem, where a graph is randomised by repeatedly making small changes to it. Even though several Markov chains have been shown to converge to the uniform distribution on their state space \cite{Rao1996,Artzy-Randrup2005,Verhelst2008,CarstensPhysRevE}, the main question for both theoreticians and practitioners remains unanswered: that is, in all but some special cases it is unknown how many changes need to be made, i.e. how many steps the Markov chains needs to take, in order to sample from a distribution that is close to uniform. 

The best known Markov chain approach for sampling graphs with fixed degree sequence is the \emph{switching model}\footnote{Also known as rewiring, switching chain and swapping edges.} \cite{Ryser1957,Taylor1981,Rao1996,Maslov2002}. It finds an approximately uniform sample of bipartite graphs, undirected graphs or directed graphs with given vertex degrees, by repeatedly switching the ends of non-adjacent edge pairs. This simple yet flexible approach converges to the uniform distribution if implemented correctly. Furthermore, this chain was proven as \emph{fully polynomial almost uniform sampler} for the following classes of graphs: regular, half-regular and irregular with bounded degrees \cite{Cooper2012,Greenhill2011,Miklos2013,Greenhill2015,Erdos16}. However, even for these classes of graphs, the theoretically proven mixing time is much too large to use in practice, e.g. $O(d^{24} n^9 log(n))$ for regular graphs with degree $d$ \cite{Greenhill2011}. Notice that the fully polynomial uniform sampler of Jerrum et al.~\cite{JerrumSinclairVigoda04} for perfect matchings can be used to sample all graphs with fixed degree sequence in polynomial time in transforming the fixed degree sequence problem in a perfect matching problem via an approach of Tutte \cite{Tutte52}. Bez\'{a}kov\'{a} et al introduced a chain extending the idea of Jerrum et al \cite{Bezakova2007}. However, the theoretical proven mixing times are much too large in practice and furthermore, this approach is more difficult to implement.

In this paper we analyse and further develop a \emph{different} Markov chain approach: the Curveball algorithm \cite{Strona2014}, which randomises bipartite graphs and directed graphs with self-loops. Experimentally, this chain has been shown to mix much faster than the corresponding switching chain \cite{Strona2014}. The intuition behind the Curveball algorithm mixing faster than the switching model can be understood when thinking of both algorithms as games in which kids trade cards. That is, think of the Curveball algorithm as an algorithm that randomises the binary $n \times m$ bi-adjacency matrix of a bipartite graph. Imagine that each row of the adjacency matrix corresponds to a kid, and the $1$'s in each row correspond to the cards owned by the kid. Then at each step in the Curveball algorithm, two kids are randomly selected, and trade a number of their differing cards. Using this same analogy for the switching model, in each step two cards are randomly selected and traded if firstly they are different and secondly they are owned by different kids. Intuitively, the Curveball algorithm is clearly a more efficient approach to randomise the card ownership by the kids. More formally, the Curveball algorithm is also based on switches but instead of making one switch, several switches can be made in a single step. We show that this leads to possibly exponentially many graphs being reached in a single step, in contrast with the switching model where at most $O(n^4)$ (the maximum number of possible edge pairs) graphs can be reached in a single step.

Several algorithms closely related to the Curveball algorithm were discovered independently by Verhelst \cite{Verhelst2008}. In particular, Verhelst already made the critical change from switches to trades. The Curveball algorithm is briefly mentioned by Verhelst as a variation on his non-uniform sampling algorithms. However, he prefers a Metropolis-Hastings approach to obtain uniform samples, since intuitively it mixes faster. It is unclear if the added complexity of a single trade in this algorithm causes the overall algorithm to run faster. Verhelst furthermore introduces an algorithm similar to the Curveball algorithm that fixes the position and number of self-loops, and hence can be used to randomise directed graphs\footnote{Throughout this paper we use the convention that directed graphs do not contain self-loops or multiple edges.}.

Here, we propose two extensions of the Curveball algorithm: the Directed Curveball algorithm, which samples directed graphs and the Undirected Curveball algorithm, which samples graphs\footnote{Throughout this paper we use the convention that graphs do not contain self-loops or multiple edges.}. Our proposed algorithm for directed graphs differs from Verhelst's algorithm in the way it deals with induced cycle sets \cite{Berger2010}. By introducing these extensions, we show that, just like the switching model, the Curveball algorithm offers a flexible framework that can be used to randomise several classes of graphs. 

Furthermore, we propose a modification to the Curveball algorithm and the Directed Curveball algorithm, that further increases the number of states that can be reached in a single step. We refer to these algorithms as the Global Curveball algorithm and the Global Directed Curveball algorithm respectively. In the card game analogy, our modification corresponds to letting \emph{all} kids trade cards in pairs \emph{simultaneously} instead of letting only one pair of kids trade. 

We prove that both extensions of the Curveball algorithm, as well as our global directed Curveball algorithms, converge to the uniform distribution. We do so by showing that their Markov chains are ergodic (the underlying state graph is non-bipartite and connected) and the transition probabilities are symmetric (see \cite{Jerrum2003} for an overview on random sampling). Our proofs follow the approach in \cite{CarstensPhysRevE} where the original Curveball algorithm was proven to converge to the uniform distribution. 

We show experimentally that the introduced Curveball algorithms all tend to mix magnitudes faster than the respective switching models. However, even though experimentally it is clear that the Curveball algorithm outperforms the switching model, we do not have a theoretical justification for this. In fact, it turns out that the techniques used to prove fast mixing for the switching chain can not be transferred to the Curveball algorithm. Hence we present the question of fast mixing for Curveball algorithms as an interesting open problem. In our opinion, the Curveball algorithm provides a big opportunity and step forward to fast mixing Markov chains for the sampling of graphs with fixed degree sequence. 

The remainder of this paper is organised as follows. Section \ref{sec:CB_ext} first discusses the original Curveball algorithm in terms of adjacency lists, it then introduces two extensions of the Curveball algorithm: the Directed Curveball algorithm that randomises directed graphs, and the Undirected Curveball algorithm that randomises graphs. Furthermore it introduces our modification to the Curveball algorithm and Directed Curveball. In Section \ref{sec:unbiased} we prove that under mild conditions, all proposed algorithms converge to the uniform distribution. Section \ref{sec:mixing_times} presents our experimental results on the run-times of all proposed algorithms. Furthermore we analyse why the proof of rapid mixing for the switching model can not be used for the Curveball algorithm. Finally we discuss our conclusions and recommendations for further research in Section \ref{sec:concl}.

\section{The Curveball algorithm and its extensions}
\label{sec:CB_ext}

We start with a formal definition. Given two lists $(a_1, \dots, a_n)$ and $(b_1, \dots, b_{n'})$ of non-negative integers, the \emph{realization problem for bipartite fixed degree sequence} asks whether there is a labelled bipartite graph, $G=(V, U, E)$, such that all vertices $v_1, \dots,v_n$ with $v_i \in V$ have degree $a_i$ and all vertices $u_1,...,u_{n'}$ with $u_i\in U$ degree $b_i$. Analogously, the \emph{realisation problem for directed fixed degree sequence} asks for a labelled directed graph for a list $(a_1,b_1),\dots,(a_n,b_n)$, and the \emph{realisation problem for undirected fixed degree sequence} for a graph for given list $(a_1,\dots,a_n).$ For an overview about these problems see~\cite{Berger2012}. The corresponding graphs or lists for each problem are called \emph{realisations} or \emph{degree sequences}, respectively.

The Curveball algorithm, as introduced in \cite{Strona2014,CarstensPhysRevE}, is a Markov chain approach to the uniform random sampling of a realisation with bipartite fixed degree sequence.  Given one such realization, the Curveball algorithm finds others by repeatedly making small changes to the \emph{adjacency list representation of the bipartite graph.} 

% \subsection{Adjacency list representation}
The adjacency list representation \cite{Jungnickel1999} of a bipartite graph $G=(V,U,E)$ is a set of lists $A_i$, one for each vertex $v_i \in V$. The list $A_i$ contains the indices $j$ corresponding to the neighbours $u_j$ of $v_i$ (see Figure \ref{fig:adj_lists}(a))\footnote{A bipartite graph can also be represented by sets $A_i$ corresponding to neighbours of the vertex $u_i$. Depending on the degree sequence, the Curveball algorithm may run faster on this representation.}. The adjacency list representation of all graphs discussed in this paper are in fact sets of \emph{sets}. We will therefore from now on refer to this representation as the \emph{adjacency set representation}. 

\begin{figure}[htb]
\centerline{\includegraphics[width=120mm]{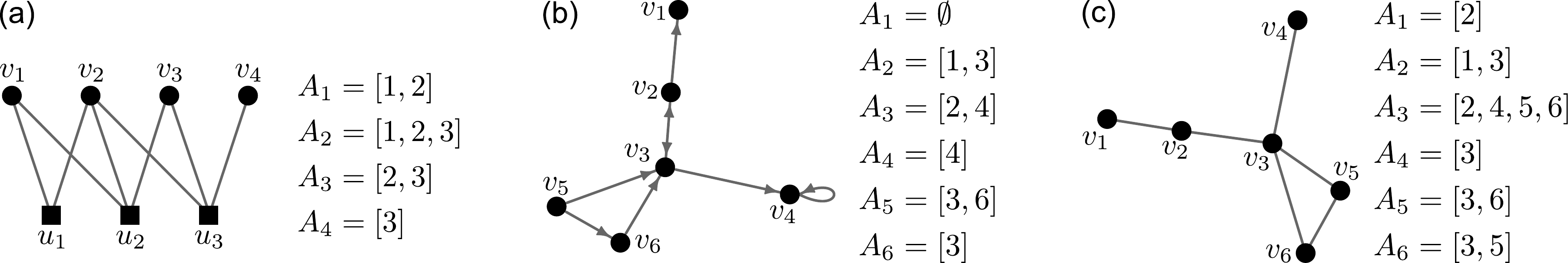}}
\caption{\label{fig:adj_lists} The adjacency set representations of (a) a bipartite graph, (b) a digraph, and (c) a graph.}
\end{figure}

The Curveball algorithm randomises the adjacency sets of a bipartite graph using the following steps. (a) Select two sets $A_i$ and $A_j$ at random. (b) Let $A_{i-j}$ be all indices that are in $A_i$ but not in $A_j$, i.e. $A_{i-j} = A_i \setminus A_j$. Similarly define $A_{j-i} = A_j \setminus A_i$. (c) Create new sets $B_i$ by removing $A_{i-j}$ from $A_i$ and adding the same number of elements randomly chosen from $A_{i-j} \cup A_{j-i}$. Combine $A_j \setminus A_{j-i}$ with the remaining elements of $A_{i-j} \cup A_{j-i}$ to form $B_j$. (d) Reiterate step (a)-(c) $N$ times, for a certain fixed number $N$. 

We follow \cite{CarstensPhysRevE} and refer to one iteration of steps (a)-(c) as a \emph{trade} and the number of exchanged indices $|B_i \setminus A_i | = |B_j \setminus A_j |$ as the \emph{size of the trade}. Notice that trades can be of size zero and such trades correspond to repeating the current state in the Markov chain. Furthermore, note that each switch in the switching model for directed graphs equals a trade of size one in the Curveball algorithm, as was shown in \cite{CarstensCompleNet2014,Verhelst2008}. However, the Curveball algorithm in addition allows trades of larger size. 

% Secondly, trades of size one correspond exactly to switches in the switching chain. Thirdly, trades of size larger than one correspond to making several switches \emph{in a single step} of the Markov chain. In particular, all switches of the switching model correspond to trades in the Curveball algorithm, but the Curveball algorithm allows additional trades.  

As discussed in \cite{CarstensPhysRevE}, the Curveball algorithm can be used to sample directed graphs with at most one self-loop per vertex (without multiple edges) with fixed in- and out-degrees. The adjacency set representation of a digraph with self-loops consists of sets $A_i$ corresponding to the out-neighbours of a vertex $v_i$ (Figure \ref{fig:adj_lists}(b))\footnote{It is also possible to use a representation based on \emph{in-neighbours}. In this case, the sets $A_i$ correspond to the in-neighbours of $v_i$. Depending on the digraph that is being randomised, the Curveball algorithm may be more efficient when using the in-neighbour representation.}.

We now introduce two extensions of the Curveball algorithm: The \emph{Directed Curveball algorithm} randomises directed graphs for a fixed degree sequence $S $ $=$ $(a_1$, $b_1)$, $ \dots$, $(a_n, b_n)$ (realisation problem for directed fixed degree sequence) and the \emph{Undirected Curveball algorithm} randomises graphs for a fixed degree sequence $S = (a_1,...,a_n)$ (realisation problem for undirected fixed degree sequence). %\todo{Both versions are based on the application of simultaneous switches insuring a fixed amount of adjacent neighbors for each vertex. Hence the vertex degrees remain constant.} 

\subsection{The Directed Curveball algorithm}
\label{sec:CB_simple}
The Directed Curveball algorithm randomises directed graphs by randomising their adjacency set representation. The adjacency set representation of a digraph is the same as that of directed graphs with self-loops, except that it has the property $i \notin A_i$ for all $i$, since directed graphs do not contain self-loops. The \emph{Directed Curveball algorithm} differs from the Curveball algorithm in step (b) only. In the Directed Curveball algorithm, the set $A_{i-j}$ is defined as all elements in $A_i$ not in $A_j$ \emph{and not equal to $j$}, i.e. $A_{i-j} := A_i \setminus (A_j \cup \{j\})$. The set $A_{j-i}$ is defined analogously by $A_{j-i} := A_j \setminus (A_i \cup \{i\})$. This small change ensures no self-loops are introduced while randomising directed graphs. Figure \ref{fig:simple_cb} illustrates the Directed Curveball algorithm. 

Notice that for the Directed Curveball algorithm trades can be of size zero and such trades correspond to repeating a state in the Markov chain. The Lemma below shows that all switches in the switching model for directed graphs equal trades of size one in the Directed Curveball algorithm. But, the Directed Curveball algorithm in addition allows trades of larger size. 

\begin{lemma}
\label{lem:subgraph_simple_dir}
Any switch in a digraph is a trade of size one in step (c) of the Directed Curveball algorithm. 
\begin{proof} 
Let $(x,y)$ and $(u,v)$ be arcs in a digraph $G$ that are allowed to be switched. Then $x$ can not be equal to $v$ and $u$ can not be equal to $y$ since otherwise this switch would introduce a self-loop. Furthermore $v \notin A_x$ and $y \notin A_u$ since otherwise the resulting digraph would have multiple edges. In particular this implies $y \in A_{x-u}$ and $v \in A_{u-x}$. Now if row $x$ and row $u$ are selected for a trade, then $B_x = (A_x \setminus \{y\}) \cup \{v\}$ and $B_u = (A_u \setminus \{v\}) \cup {y}$ are possible sets in step (c) that lead to exactly the two new arcs $(x,v)$ and $(u,y)$.
\end{proof}
\end{lemma}

\begin{figure}[htb]
\centerline{\includegraphics[width=120mm]{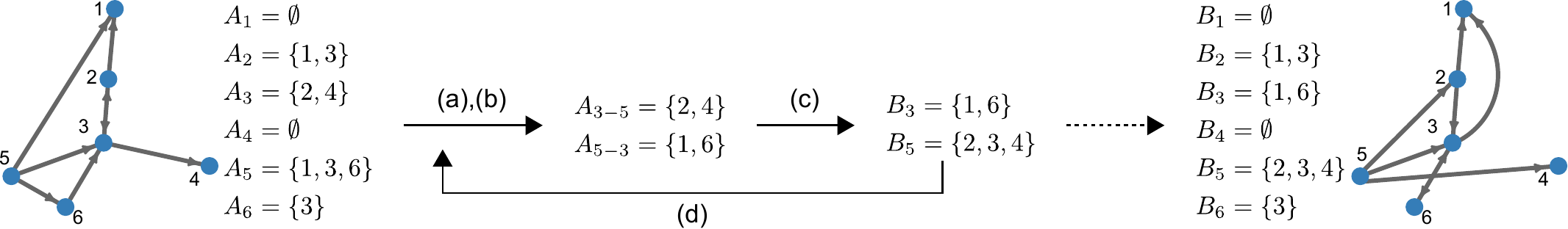}}
\caption{\label{fig:simple_cb}Illustration of the Directed Curveball algorithm. In this example, the vertices $v_3$ and $v_5$ are selected in step (a). In step (b) their entries are compared to create the sets $A_{3-5}$ and $A_{5-3}$. Notice that 3 is removed from $A_5$ since trading it would introduce a self-loop at vertex $v_3$. In step (c) the new sets $B_3$ and $B_5$ are constructed by randomly redistributing the elements $1,2,4$ and $6$. Step (d) repeats steps (a)-(c) $N$ times. The dashed arrow shows the result of the single trade made in steps (a)-(c).}
\end{figure}

\subsection{The Undirected Curveball algorithm}
\label{sec:CB_simple_undir}
The Undirected Curveball algorithm samples graphs with fixed degree sequence. The adjacency set representation of a graph $G=(V,E)$ is a set of sets $A_i$. The set $A_i$ now contains the indices of the neighbours of vertex $v_i$ (Figure \ref{fig:adj_lists}(c)). The symmetry of a graph is reflected in its adjacency set representation. That is, $i$ is an element of $A_j$ if and only if $j$ is also an element of $A_i$. Furthermore, these sets have the property that $i \notin A_i$ for all $i$, since graphs do not contain self-loops. We introduce the Undirected Curveball algorithm. This algorithm randomises the adjacency set representation of a graph while maintaining its symmetry and ensuring no self-loops are introduced. 

The Undirected Curveball algorithm is defined by the following steps. (a) Randomly select two sets $A_i$ and $A_j$. (b) Let $A_{i-j}$ be the set of elements in $A_i$ not in $A_j$ and not equal to $j$, i.e. $A_{i-j} := A_i \setminus (A_j \cup \{j\})$. Analogously define $A_{j-i} := A_j \setminus (A_i \cup \{i\})$. (c) Create a new set $B_i$ by removing $A_{i-j}$ from $A_i$ and adding the same number of elements randomly chosen from $A_{i-j} \cup A_{j-i}$. Combine $A_j \setminus A_{j-i}$ with the remaining elements of $A_{i-j} \cup A_{j-i}$ to form $B_j$. (c$^\prime$) For each index $k \in B_i \setminus A_i$, replace $j$ by $i$ in $B_k$, similarly for each $l \in B_j \setminus A_j$, replace $i$ by $j$ in $B_l$. (d) Reiterate step (a)-(c$^\prime$) $N$ times, for a certain fixed number $N$. Figure \ref{fig:simple_undir_cb} illustrates the Undirected Curveball algorithm. 

Step (b) ensures no self-loops are introduced. Step (c$^\prime$) is well-defined since $k \in B_i \setminus A_i$ implies $k \notin A_i$ and $k \in A_j$. This in turn implies that $i \notin A_k$ and $j \in A_k$ by symmetry of the adjacency sets. Thus we can replace $j$ by $i$ in $A_k$ to obtain $B_k$. Similarly $l \in B_j \setminus A_j$ implies that $i$ is an element of $A_l$ and that $j$ is not. And thus, replacing $i$ by $j$ in $A_l$ is well-defined. Step (c$^\prime$) thus ensures that $B$ represents a graph.  

\begin{figure}[!htb]
\centerline{\includegraphics[width=120mm]{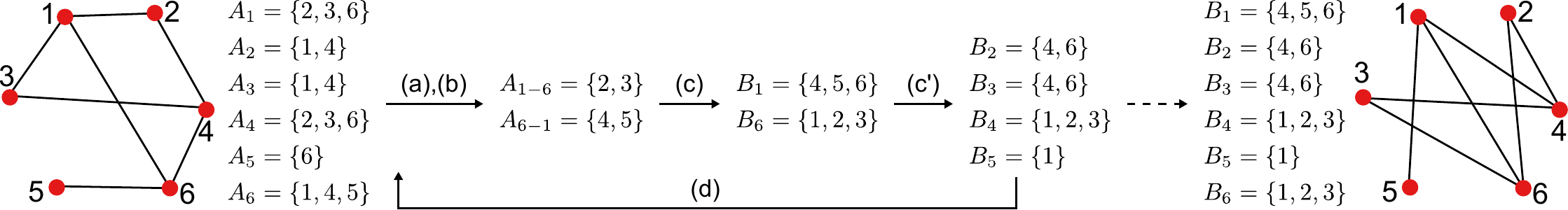}}
\caption{\label{fig:simple_undir_cb}Illustration of the Undirected Curveball algorithm. In this example the sets $A_1$ and $A_6$ are selected in step (a). In step (b) their entries are compared to create the sets $A_{1-6}$ and $A_{6-1}$. In step (c) the new sets $B_1$ and $B_6$ are constructed by randomly redistributing the elements $2,3,4$ and $5$. Step (c') updates the sets corresponding to indices involved in the trade. In this case we need to update $B_2$ and $B_3$ by removing $1$ and inserting $6$, conversely $B_4$ and $B_5$ are updated by removing $6$ and inserting 1. Step (d) repeats steps (a)-(c') $N$ times. The dashed arrow shows the result of the single trade illustrated by steps (a)-(c').}
\end{figure}

Notice that the Undirected Curveball algorithm includes trades of size zero which correspond to repeating the current state in the Markov chain. Furthermore, Lemma \ref{lem:switch_equals_trade_undir} shows that any switch in the switching model for graphs corresponds to a trade of size one in the Undirected Curveball algorithm. In fact, Figure \ref{fig:undirCB} shows that for each switch in the switching model, there are two different trades of size one in the Undirected Curveball algorithm.

\begin{lemma}
\label{lem:switch_equals_trade_undir}
Let $G, G'$ be graphs that differ by a switch. There are two trades of size one in the Undirected Curveball algorithm from $G$ to $G'$. 
\begin{proof}
Without loss of generality we may assume that $G=(V,E)$ and $G'=(V,E')$ differ by a switch from $\{x,y\}$ and $\{u,v\}$ to $\{x,v\}$ and $\{u,y\}$ (see Figure \ref{fig:undirCB}). Let $\{A_1,\dots,A_n\}$ be the adjacency set representation of $G$, then $y \in A_{x-u}$ since the edge $\{x,y\}$ is an edge of $G$, the edge $\{u,y\}$ is not and $y$ can not be equal to $u$ since $\{u,y\} \in E'$ and $G'$ has no self-loops. Similarly we find that $v \in A_{u-x}$ and hence the trade that swaps $y$ and $v$ between rows $x$ and $u$ results in the graph $G'$. Similarly, there is a second trade which generates $G'$, namely the trade that exchanges $x$ and $u$ between sets $A_y$ and $A_v$.
\end{proof}
\end{lemma}
  
\begin{figure}[!htb]
\center
\includegraphics[width=120mm]{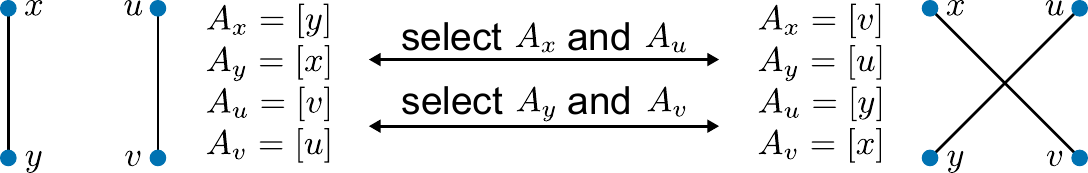}%
\caption{\label{fig:undirCB}A switch corresponds to a trade of size one in the Undirected Curveball algorithm. Notice that each switch can be realized by two distinct trades: the switch from $\{x,y\}$ and $\{u,v\}$ to $\{x,v\}$ and $\{u,y\}$ can be realized by selecting $A_x$ and $A_u$ or by selecting $A_y$ and $A_v$. }
\end{figure}

Analogous to the other versions of the Curveball algorithm, the Undirected Curveball algorithm in addition allows trades of larger size, corresponding to making several switches at once.

\subsection{The global directed Curveball algorithms}\label{sec:modified_curveball}
We now introduce the Global Curveball algorithm and the Global Directed Curveball algorithm as a modification of the Curveball algorithm and the Directed Curveball algorithm respectively. The number of graphs that can be reached by a single step in the Markov chain of these \emph{global} Curveball algorithms is even higher than for the regular Curveball algorithms. This modification is motivated by our desire to improve the Curveball algorithm in situations where all trades correspond to switches, i.e. larger trades cannot happen. This happens for instance when a bi-adjacency matrix represents a perfect matching, i.e. row and column sums are one. In this case, each trade has at most size one, and hence corresponds to either a switch or a repeated state. 

As explained in the introduction, instead of attempting trades between two sets (kids) in the adjacency set representation, the global algorithms allow each of the sets to trade in pairs. More formally, the Global Curveball algorithm and the Global Directed Curveball algorithm are defined as follows. We replace step (a) in the Curveball algorithm and Directed Curveball algorithm by the following step: For all lists $A_1, \dots A_n$ choose uniformly at random a 2-partition $(A_{i_{1}},A_{i_{2}})$, $(A_{i_{3}},A_{i_{4}})$, $\dots$, $(A_{i_{n-1}},A_{i_{n}})$ for even $n$, or, $(A_{i_{1}}, A_{i_{2}})$, $(A_{i_{3}},A_{i_{4}})$, $\dots$, $(A_{i_{n-2}},A_{i_{n-1}})$, $(A_{i_{n}})$ for odd $n$. For each pair $(A_{i_{k}},A_{i_{k+1}})$ apply steps (b)-(c). Step (d) again reiterates steps (a)-(c) $N$ times. We refer to one iteration of steps (a)-(c) as a \emph{global trade} in analogy to the term trade for the Curveball algorithm. In the Appendix we present Algorithm~\ref{AlgorithmTwoPartition} for determining a uniform $2$-partition in $O(n)$ asymptotic runtime. 

In our example of a perfect matching above, the Global Curveball algorithm has an exponential number of possible transitions for each realisation: each pair of the $(n-1)\cdot (n-3) \cdots 3 \cdot 1$ possible partitions allows $2^{\nicefrac{n}{2}}$ different global trades, since all $\nicefrac{n}{2}$ adjacency list pairs allow two trades (a switch or a trade of size zero). This exponential number of possible transitions is in contrast to only a quadratic number of transitions in the Curveball algorithm: each transition corresponding to a single switch for each of the $\binom{n}{2}$ possible adjacency list pairs.

A similar modification to the Undirected Curveball algorithm is not possible. A trade between two sets affects additional sets in step ($c'$), and hence a trade between a pair of sets is not independent of trades between other pairs of sets. 

\section{Convergence to the uniform distribution}
\label{sec:unbiased}

A Markov chain can be seen as a \emph{random walk} \cite{Lovasz1996} on a set $\Omega$ of combinatorial objects, the so-called \emph{states}. Two states $x,y \in \Omega$ are connected via a transition edge $(x, y)\in \Psi$, when $x$ can be transformed into $y$ via a small local change. For the switching chain such a `local change' corresponds to a switch, and in the Curveball algorithms to one trade. For both algorithms, the states of $\Omega$ are all realisations of a fixed degree sequence. This definition induces a so-called directed \emph{state graph} $\Gamma = (\Omega,\Psi)$, representing the states and how they are connected by local changes. A step from $x$ to $y$ in a random walk is done with \emph{transition probability} $p_{xy}.$  

In \cite{CarstensPhysRevE} the Curveball algorithm was proven to converge to the uniform distribution by applying the fundamental theorem for Markov chains (see for example \cite{Levin2009}).
 
% The key to proving that the Curveball algorithm has an irreducible Markov chain, is to realize that the state graph of the switching model is always a subgraph of the state graph of the Curveball algorithm. Hence irreducibility for the switching model implies irreducibility for the Curveball algorithm. The Markov chain corresponding to the Curveball algorithm is trivially aperiodic since a

\begin{thm}
\label{thm:MC_stationary}
A finite Markov chain converges to its unique stationary distribution if its state graph $\Gamma = (\Omega,\Psi)$ is connected and non-bipartite. If there exists a probability distribution $\pi:\Omega \mapsto [0,1]$ such that the detailed balanced equations, $\pi(x)p_{xy}=\pi(y)p_{yx}$, are satisfied for all $(x,y) \in \Psi,$ then $\pi$ is the unique stationary distribution.
\end{thm}

Markov chains which fulfil these properties are called \emph{ergodic}. This theorem implies that an ergodic Markov chain converges to the uniform distribution if $p_{xy} = p_{yx}$ for all $x,y \in \Omega$. 

We now derive the conditions for which the Directed Curveball algorithm, the Undirected Curveball algorithm, the Global Curveball algorithm, and the Global Directed Curveball algorithm converge to the uniform distribution on their respective state spaces, i.e. the set of all possible solutions of the realisation problem. 

\subsection{Directed Curveball algorithm}
\label{sec:unbiased_simple}
We start by deriving the transition probabilities of the Directed Curveball algorithm. 

\begin{lemma}
\label{lem:prob_simple_dir}
Let $A$ and $B$ be two adjacency set representations of directed graphs with equal degree sequence. The transition probability $P_{AB}$ from $A$ to $B$, in the Directed Curveball algorithm, is given by 
\[ P_{AB} = \begin{cases} 
		\frac{2}{n(n-1)} \frac{s_i!s_j!}{(s_i + s_j)!} & \mbox{if } B \mbox{ only differs from } A \mbox{ in sets } A_i \mbox{ and } A_j,\\
		1 - \sum_{C, C \neq A} P_{AC} & \mbox{if } A = B, \\
		0            & \mbox{otherwise.} 
		\end{cases}\]   
where $s_i = |A_{i-j}|$ and $s_j = |A_{j-i}|$. Hence, $P_{AB} = P_{BA}$ for all $A, B$. 
\begin{proof}
The probability of transitioning from a state $A$ to another state $B$ that differs in a trade between sets $A_i$ and $A_j$ can be found as follows. The probability of selecting set $A_i$ and set $A_j$ equals the inverse of the number of pairs of sets in the adjacency set representation, i.e. $\nicefrac{2}{n(n-1)}$, where $n$ equals the number of sets. The probability that shuffling $A_{i-j} \cup A_{j-i}$ results in state $B$ equals the inverse of the number of ways you can select $s_i$ unordered elements from the set $A_{i-j} \cup A_{j-i}$ in step (c) of the algorithm. This probability equals $\nicefrac{s_i!s_j!}{(s_i+s_j)!}$ since $|A_{i-j} \cup A_{j-i}| = s_i + s_j$. 

To show that the probabilities $P_{AB}$ and $P_{BA}$ are equal for all adjacency sets $A$ and $B$ we only need to show that this is true in the non-trivial case when the adjacency sets differ exactly in two sets, say $A_i$ and $A_j$. 
$|B_i|$ equals $|A_i|$, and $|B_j|$ equals $|A_j|$ since trades do not change the number of elements. We find $|A_{i-j}|=|B_{i-j}|$ since $A_{i-j}=A_i \setminus \{A_j \cup \{j\}\}$, and $B_{i-j}=B_i \setminus \{B_j \cup \{j\}\}$. Similarly $|B_{j-i}| = |A_{j-i}|$, and indeed we find $P_{AB} = P_{BA}$.
% Now $B_i \cap B_j$ equals $A_i \cap A_j$ since a trade only involves elements in which the sets differ. Furthermore, $|B_i|$ equals $|A_i|$, and $|B_j|$ equals $|A_j|$ since trades do not add or remove elements. Hence $|B_{i-j}| = |A_{i-j}|$, since $B_{i-j} = B_i \subset B_j = B_i \subset (B_i \cap B_j)$. Similarly $|B_{j-i}| = |A_{j-i}|$, and indeed we find $P_{AB} = P_{BA}$.} 
\end{proof}
\end{lemma}

We next discuss that connectance of Theorem \ref{thm:MC_stationary} is fulfilled for the Directed Curveball algorithm. Notice that Lemma \ref{lem:subgraph_simple_dir} implies that the state graph of the switching model for simple directed graphs is a subgraph of the state graph of the Directed Curveball algorithm because each switch is a trade of size one. Hence, since both Markov chains have the same state space, whenever the switching model for directed graphs has irreducible Markov chain then so does the Markov chain of the Directed Curveball algorithm. This leads to the following theorem. 

\begin{thm}\label{thm:concergence_directed_notConnected}
If the state graph corresponding to the switching chain for directed fixed degree sequences is connected, then the Markov chain of the Directed Curveball chain converges to its stationary distribution, which is the uniform distribution. 
\begin{proof}
The state graph of the switching model with respect to directed graphs is a subgraph of the state graph of the Directed Curveball algorithm (Lemma \ref{lem:subgraph_simple_dir}). Hence, a connected state graph of the switching chain implies a connected state graph of the Directed Curveball chain. The state graph of the Directed Curveball chain is always non-bipartite, since there is a non-zero transition probability $P_{AA}$ of repeating each state $A$ in step (c), due to trades of size zero. Finally $P_{AB} = P_{BA}$ for all states $A$ and $B$ (see Lemma \ref{lem:prob_simple_dir}). Hence convergence to the uniform distribution follows from Theorem \ref{thm:MC_stationary}.
\end{proof}
\end{thm}

It is well-known that the switching model for directed graphs can have a reducible Markov chain \cite{Rao1996}. The simplest example being a directed cycle on three vertices, its opposite orientation can not be achieved by switches, since no switch is possible without introducing self-loops. We know of two approaches to mitigate this problem for switching chains: one is to introduce an additional move which reorients directed cycles of length three (hexagonal move in \cite{Rao1996}). This is the approach taken in \cite{Verhelst2008} to sample directed graphs. However, we follow a second approach that uses a pre-sampling step \cite{Berger2010}. We prefer this approach because the corresponding Markov chain runs on a (potentially much) smaller state graph compared with the triangle reorientation chain and hence should be faster. Furthermore, this approach is much easier to transfer to the Directed Curveball algorithm. 

To discuss this approach we need the definition of \emph{induced cycle sets} \cite{Berger2010} for a directed graph sequence $S$. An induced cycle set consists of three indices, $i_1$, $i_2$ and $i_3$, for pairs in $S$ such that the vertices $v_{i_{1}},v_{i_{2}},v_{i_{3}}$ form a directed cycle in \emph{each} directed graph realisation of $S$. 

Let $\Psi_S$ be the state graph of the switching model for directed graphs with fixed degree sequence $S$. Berger et al \cite{Berger2010} prove that $\Psi_S$ is non-connected if and only if $S$ contains an \emph{induced cycle set}. In fact, they prove that if $S$ contains $k$ induced cycle sets, then $\Psi_S$ consists of $2^k$ isomorphic components where each component corresponds to a specific orientation for all $k$ cycles. Hence, instead of introducing a triangle reorientation, Berger et al. \cite{Berger2010} choose one of the isomorphic components uniformly at random prior to running the switching model on this component. Notice that they also show that all induced cycle sets are disjunctive, i.e. at most $\nicefrac{n}{3}$ such cycle sets are possible.

\begin{thm}
The state graph of the Directed Curveball algorithm decomposes in $2^k$ isomorphic components, where $k \leq n$ is the number of induced cycle sets. If it is not connected $(k>0)$, then applying the Directed Curveball algorithm on any component leads to the uniform distribution of all states in this component.
\begin{proof}
The state graph for the Directed Curveball decomposes in at most $2^k$ isomorphic components because the state graph of the switching chain for directed sequences is a subgraph (Lemma~\ref{lem:subgraph_simple_dir}) of the state graph of the Directed Curveball algorithm, and the state graph of the switching chain decomposes in $2^k$ isomorphic components \cite{Berger2010}. All these components contain realisations which only differ in the orientation of triangles of induced cycle sets. The Directed Curveball algorithm basically applies switches, and is not able to change these triangles. Hence, the state graph of the Directed Curveball chain has exactly $2^k$ components, consisting of exactly the same states as the components of the switching chain state graph. Therefore trades must be identical in each component leading to isomorphic components. Using the non-bipartiteness of each component (proof of Theorem~\ref{thm:concergence_directed_notConnected}) and $P_{AB}=P_{BA}$ with Lemma~\ref{lem:prob_simple_dir}, it follows with Theorem~\ref{thm:MC_stationary} that the Directed Curveball algorithm on a component converges to the uniform distribution on all states in the component. 
\end{proof}
\end{thm}

This theorem implies that choosing one component uniformly at random and applying the Directed Curveball algorithm on this component leads to a uniform distribution of all states. Hence, we propose the following Adjusted Directed Curveball algorithm for a directed graph $G$ with degree sequence $S$: (1) Identify all $k$ induced cycle sets in $S$. (2) Choose a random orientation for each induced cycle set in $G$, leading to a realisation $G'$ of $S$. (3) Use the Directed Curveball algorithm starting with $G'$. 

\begin{cor}
Let $G$ be a directed graph with directed graph sequence $S$. The Adjusted Directed Curveball algorithm converges to the uniform distribution on all directed graph realisations of $S$. \qed
\end{cor}

We now propose a linear-time algorithm for the identification of all induced cycle sets in (1) of the Adjusted Directed Curveball algorithm. This approach follows a result of LaMar~\cite{LaMar2009}, which we describe in a different form and simplify. We first define the \emph{corrected Ferrers matrix} for a given degree sequence. Let $S:=(a_1,b_1)$, $\dots$, $(a_n,b_n)$ be a degree sequence in non-increasing lexicographical order. The $n \times n$ \emph{corrected Ferrers matrix} $F$ corresponding to $S$ is a matrix with row sums $b_1,\dots,b_n$. Each row $i$ consists of $b_i$ consecutive $1$'s followed by consecutive $0$'s with the exception that the diagonal elements $F_{ii}$ are always $0.$ This leads to column sums $f_1,\dots,f_n$ of $F.$ The classical result of Chen-Fulkerson-Ryser states that $S$ has a realisation if and only if $\sum_{i=1}^l (f_i-a_i)\geq 0$ for all $l \in \{1,\dots,n\}.$ For a comprehensive discussion we recommend the paper of Berger \cite{Berger2014}. We define $1_{x}:\mathbb{N} \mapsto \mathbb{N}$ as the function with $1_{x}(y)=1$ for $y=x$ and $1_{x}(y)=0$ in all other cases. LaMar~\cite{LaMar2009} stated the following result.

\begin{thm}[LaMar 2009, \cite{LaMar2009}] \label{TheoremLaMar}
Let $S=(a_1,b_1),\dots,(a_n,b_n)$ be a lexicographical non-increasing degree sequence with a directed graph as realisation, and $f_1,\dots,f_n$ the column sums of its corrected Ferrers matrix.\\
Let $\overline{S}=(b'_1,a'_1),\dots,(b'_n,a'_n)=(b_{\sigma(1)},a_{\sigma(1)}),\dots,(b_{\sigma(n)},a_{\sigma(n)})$ be a permutation of $S$ which was generated by exchanging the component order in all pairs and sorting it in non-increasing lexicographical order. Let $f'_1,\dots,f'_n$ be the column sums of its corrected Ferrers matrix. \\ Indices $i,i+1,i+2$ form an induced cycle set in $S$ if and only if 
\begin{enumerate}
\item $(a_i,b_i)=(a_{i+1},b_{i+1})=(a_{i+2},b_{i+2})=(k,i),$
\item $(b'_k,a'_k)=(b'_{k+1},a'_{k+1})=(b'_{k+2},a'_{k+1})=(i,k),$
\item $\sum_{i=1}^l (f_i-a_i)=1_i(l)+1_{i+1}(l)$ for $l \in \{i-1,\dots,i+2\},$
\item $\sum_{i=1}^l (f'_i-b'_i)=1_k(l)+1_{k+1}(l)$ for $l \in \{k-1,\dots,k+2\}.$
\end{enumerate} 
\end{thm}

We state a simpler version of this theorem which is based on the observation that items (2) and (4) follow directly from items (1) and (3).

\begin{restatable}{thm}{TheoremInducedCycles}
\label{TheoremInducedCycles}
Let $S=(a_1,b_1),\dots,(a_n,b_n)$ be a lexicographical non-increasing degree sequence with a directed graph as realisation, and $f_1,\dots,f_n$ the column sums of its corrected Ferrers matrix. Indices $i,i+1,i+2$ form in $S$ an induced cycle set if and only if 
\begin{enumerate}
\item $(a_i,b_i)=(a_{i+1},b_{i+1})=(a_{i+2},b_{i+2})=(k,i),$
\item $\sum_{j=1}^l (f_j-a_j)=1_i(l)+1_{i+1}(l)$ for $l \in \{i-1,\dots,i+2\},$
\end{enumerate} 
\end{restatable}
\begin{proof}
The proof is given in the Appendix.
\end{proof}

This results in the following algorithm for detecting all induced cycle sets in linear time (which was also the case for Lamar's Theorem~\ref{TheoremLaMar}). (1) Sort $S$ in non-increasing lexicographical order, (2) determine the set $T$ of all triples $(i,i+1,i+2)$ fulfilling (1) in Theorem~\ref{TheoremInducedCycles}, (3) construct the corresponding Ferrers matrix for $S$, and (4) determine $s_l:=\sum_{i=1}^l (f_i-a_i)$ for $l \in \{1,\dots,n\}.$ If $(s_{i-1},s_i,s_{i+1},s_{i+2})=(0,1,1,0)$ for triple $t:=(i,i+1,i+2) \in T$ then $t$ is an induced cycle set. Step (1),(2),(4) can be done in $O(n)$ time. The construction of the Ferrers matrix needs $O(m)$ time where $m$ denotes the number of $1$'s in the matrix. In summary this algorithm leads to an asymptotic linear time.

\subsection{Undirected Curveball}
\label{sec:unbiased_simple_undir}
We now discuss the conditions under which the Undirected Curveball algorithm converges to the uniform distribution. We start by deriving its transition probabilities. 

\begin{lemma}
\label{lem:undir_cb_trans_prob}
Let $A$ and $B$ be two adjacency set representations of graphs with equal degree sequence. The transition probability $P_{AB}$ from $A$ to $B$, in the Undirected Curveball algorithm, is given by 
\[ P_{AB} = \begin{cases} 
		\frac{2}{n(n-1)} \left( \frac{s_i!s_j!}{(s_i + s_j)!} + \frac{s_k!s_l!}{(s_k + s_l)!} \right) & \mbox{if } A \mbox{ and } B \mbox{ differ by a trade of size one,}  \\
		& \mbox{between sets } A_i \mbox{ and } A_j, \mbox{ exchanging } k \mbox{ and }l,\\
		\frac{2}{n(n-1)} \frac{s_i!s_j!}{(s_i + s_j)!} & \mbox{if } A \mbox{ and } B \mbox{ differ in a trade of size more } \\ 
		& \mbox{than one between sets } A_i \mbox{ and } A_j,\\
		1 - \sum_{C, C \neq A} P_{AC} & \mbox{if } A = B, \\
		0            & \mbox{otherwise.} 
		\end{cases}\]   
with $s_i = |A_{i-j}|$, $s_j = |A_{j-i}|$, $s_k = |A_{k-l}|$ and $s_l = |A_{l-k}|$. In particular, $P_{AB} = P_{BA}$ for all states $A, B$.
\begin{proof}
When the adjacency sets $A$ and $B$ differ by a trade of size one between sets $A_i$ and $A_j$ involving indices $k$ and $l$, then they also differ by a trade of size one between sets $A_k$ and $A_l$ involving indices $i$ and $j$ (see Lemma \ref{lem:switch_equals_trade_undir}). Hence, we need to add the probabilities of selecting either one of these trades. When $A$ and $B$ differ in trade of size larger than one, there is a unique pair of sets that corresponds to this trade, hence we find the usual transition probability. 

To see that $P_{AB} = P_{BA}$, observe that just like in the Directed Curveball algorithm (see Lemma \ref{lem:prob_simple_dir}), a trade between sets $A_i$ and $A_j$ to form $B_i$ and $B_j$ implies that $|A_{i-j}| = |B_{i-j}|$ and $|A_{j-i}| = |B_{j-i}|$ since trades leave common elements invariant and do not alter the number of elements in each set. 
\end{proof}
\end{lemma}

\begin{thm}
For any graph $G$, the Markov chain of the Undirected Curveball algorithm starting at $G$ converges to the uniform distribution on all graphs with the same degree sequences as $G$.  
\begin{proof}
The state graph of the switching chain for graphs with fixed degree sequence is a subgraph of the state graph of the Undirected Curveball algorithm on the same states (Lemma~\ref{lem:switch_equals_trade_undir}). The state graph of the switching model was shown to be connected in \cite{Taylor1981,EggletonHolton1981} which implies the connectance of the state graph of the Undirected Curveball algorithm. The state graph of the Undirected Curveball algorithm is always non-bipartite, since there is a non-zero probability of repeating each state, due to trades of size zero. Finally $P_{AB} = P_{BA}$, see Lemma \ref{lem:undir_cb_trans_prob}. Hence by Theorem \ref{thm:MC_stationary} the Undirected Curveball algorithm converges to the uniform distribution on its state space. 
\end{proof}
\end{thm}

\subsection{The global directed Curveball algorithms}\label{sec_modified_curveball}

We start by deriving the transition probabilities for the Global Curveball algorithm of Subsection~\ref{sec:modified_curveball}. We first calculate the number of possible global trades for \emph{one} $2$-partition $P$, and then develop the number of possible transitions for \emph{all} $2$-partitions. This value will be taken as a basis for calculating the transition probabilities. In the following we denote by \emph{even partition} of a set $M:=\{1,\dots,n\}$ a $2$-partition $P=\{\{i_1,i_2\},\dots,\{i_{n-1},i_{n}\}\}$, and by \emph{odd partition} $P=\{\{i_1,i_2\},\dots,\{i_{n-2},i_{n-1}\},\{i_n\}\}$.  

\begin{lemma}\label{lem:number_global_trade}
Let $A$ be the adjacency set representations of a bipartite graph (digraph) with degree sequence $S$, and let $P$ be a $2$-partition of $M=\{1,\dots,n\}$ with $n=2k$ for even $n$, and $n=2k+1$ for odd $n$. The number $r(P)$ of global trades for $P$ in the Global Curveball chain is 
$$r(P)=\Pi_{j \in \{1,3,\dots,2k-1\}}\frac{s_{i_{j}}+s_{i_{j+1}}}{s_{i_{j}}!s_{i_{j+1}}!}$$
with $s_{i_{j}} = |A_{i_{j}-i_{j+1}}|$ and $s_{i_{j+1}} = |A_{i_{j+1}-i_{j}}|$.
\end{lemma}

\begin{proof}
The number of global trades for one partition $P$ in the Global Curveball algorithm is the product of the number of trades for each of the randomly chosen pairs $(i_{j},i_{j+1})$, since trades for these pairs are applied independently. The number of trades for each pair of rows is the same as that in the original Curveball algorithm. This number was derived in \cite{CarstensPhysRevE} and our result now follows.
\end{proof}

We now discuss an example that shows that two different $2$-partitions and  corresponding global trades may result in the same change in the adjacency set representation. 

\begin{exmp}
Let $A$ be the following adjacency set representation of a bipartite graph (digraph) with self-loops: $A_1=[1,2]$, $A_2=[2,3]$, $A_3=[1,2]$ and $A_4=[2,3]$. Consider the following two different $2$-partitions $P=\{\{1,2\},\{3,4\}\}$ and $P'=\{\{1,4\},\{2,3\}\}$. It is easy to see that the set of global trades for both partitions is exactly the same.
\end{exmp}
 
This example leads us to derive the transition probabilities of the Global Curveball algorithm as follows. 

\begin{restatable}{lemma}{ModifiedCurveballTransition}
\label{lem:modified_curveball_transitions}
Let $A$ and $B$ be two adjacency set representations of bipartite graphs (digraphs) with equal degree sequence, and $\mathcal{P}$ the set of all $2$-partitions of $M=\{1,\dots,n\}$. The transition probability $P_{AB}$ from $A$ to $B$, in the Global Curveball algorithms, is given by 
\[ P_{AB} = \begin{cases} 
		\frac{1}{|\mathcal{P}|} \sum_{\{P \in \mathcal{P}~|~B \textnormal{ results from a global trade in A using P }\}}\frac{1}{r(P)} & \mbox{if } B \mbox{ differs from } A \\ & \mbox{by a global trade},\\
		1 - \sum_{C, C \neq A} P_{AC} & \mbox{if } A = B, \\
		0            & \mbox{otherwise.} 
		\end{cases}\]   
where $r(P)$ is the number of global trades for one partition $P$ in Lemma~\ref{lem:number_global_trade}. The number $\mathcal{P}$ of $2$-partitions is given by $\Pi_{k \in \{1,3,\dots,n-1\}}n-k$ for even partitions, and by $\Pi_{k \in \{0,2,\dots,n-1\}}n-k$ for odd partitions. In particular, $P_{AB} = P_{BA}$ for all $A, B$.
\end{restatable}

\begin{proof}
We prove the result for the bipartite and directed case simultaneously. Each partition in $\mathcal{P}$ corresponds to at most one global trade (see Lemma~\ref{lem:number_global_trade}) between $A$ and $B$. The probability of selecting a partition that corresponds to a global trade between $A$ and $B$ equals $\nicefrac{1}{|\mathcal{P}|}$. For each of these partitions, $P$, the probability of selecting the corresponding global trade between $A$ and $B$ equals $\nicefrac{1}{r(P)}$. 

We derive the number of $2$-partitions $|\mathcal{P}|$ in the Appendix.

Using the same arguments as in proof of Lemma~\ref{lem:prob_simple_dir} we get $P_{AB} = P_{BA}$ for all $A, B$ because $\nicefrac{1}{r(P)}$ is the probability of $\nicefrac{n}{2}$ independent trades in the directed Curveball algorithms. 
\end{proof}

Recalling that global trades correspond to a number of independent trades in the Curveball algorithm it follows that the state graph of the directed version of the Global Curveball algorithm decomposes in $2^k$ isomorphic components whenever induced cycle sets are contained in degree sequence $S.$ All results from subsection~\ref{sec:unbiased_simple} can be applied analogously leading to an \emph{adjusted version of the global directed algorithm} which samples uniform at random one isomorphic component and uses global trades to sample within this isomorphic component uniformly at random. We do not repeat all details from subsection~\ref{sec:unbiased_simple}. Instead we state the following theorem. 

\begin{thm}\label{thm:concergence_directed_modified}
If the state graph corresponding to the switching chain for directed fixed degree sequences is connected, then the Markov chain of the Global Directed Curveball chain converges to its stationary distribution, which is the uniform distribution. On the other hand, if the Markov chain of the switching model is not connected, then applying the Global Directed Curveball algorithm to any component converges to the uniform distribution on all states in this component.
\end{thm}

\section{Mixing time and experimental stopping times}
\label{sec:mixing_times}
The most important question for practitioners as well as theoreticians is how many steps the (global) Curveball algorithms have to run from an initial probability distribution (where an initial state is taken from) to sample from a probability distribution which is close to the uniform distribution. This number is defined as the \emph{total mixing time}, i.e. the number $N$ of reiteration steps in the Curveball algorithms. 

The Curveball algorithm has experimentally been shown to run much faster than the switching algorithm \cite{Strona2014}. Although we do not know if the total mixing time of the Curveball algorithm is smaller than that of the switching model, we show that all of our proposed Curveball algorithms (see Section \ref{sec:CB_ext}) tend to run faster in experiments than the respective switching models.

\subsection{Experimental results}
We compared the mixing times of the Curveball algorithm, the Global Curveball algorithm and the switching model for a number of random and real networks. 

Our main interest is in comparing the asymptotic mixing times of these algorithms. For this reason, we want to measure the impact of the structure of the state graph on the mixing time while disregarding the impact of the different probabilities to repeat states, since the latter corresponds to a polynomial term in the asymptotic mixing time. 

In order to measure the impact of the increased number of neighbours for each graph in the Curveball algorithms as compared to the switching model, we altered the Markov chains of all algorithms slightly. Specifically, for the Curveball algorithms, when we select a row-pair for which non-zero trades exist, we ensure a non-zero trade is selected in step (c) of the algorithm, i.e. one chooses only a random subset $S$ of $A_i-j \cup A_j-i$ with $S \neq A_i-j$. In \cite{CarstensPhysRevE} it was shown that this 'Good-Shuffle Curveball algorithm' converges to the uniform distribution. It is straightforward to adjust those arguments to show that the adjusted Global Curveball algorithms still converge to the uniform distribution too. For the switching models, we adjust the algorithms such that they resemble the Curveball algorithms more closely. That is, we select a row-pair, and if non-zero trades exists we select a trade of size one at random. Again, this algorithm still converges to the uniform distribution, due to an argument similar to that for the Good-Shuffle Curveball algorithm. 

The perturbation scores \cite{Strona2014} between a current (directed) graph in the Markov chain and the initial (directed) graph computes the fraction of edges in which these two graphs differ. It hence provides a dissimilarity measure. We use the point where this perturbation score stabilizes as an estimate for the mixing time of the Markov chains. 

Figure \ref{fig:digraphs} shows our comparison of the algorithms for ten directed graphs. We randomise six \er networks $G(n,p)$ with 1000 vertices and varying probability $p \in \{0.05, 0.06, \dots 0.1\}$, three random networks generated using the simple preferential attachment model introduced by Albert and Barab\'{a}si \cite{Barabasi1999} with 1000 vertices and varying number of added edges $m \in \{1,2,3\}$ per step, and a real directed graph which represents a protein interaction network \cite{konect:2016:maayan-figeys,konect:figeys}.

Our main findings from these experiments are the following. The Directed Curveball algorithm converges much faster than the switching chain for the \er random networks and the real network. Furthermore, the Global Directed Curveball algorithm converges dramatically faster than the Directed Curveball algorithm for all networks. The Directed Curveball algorithm and switching chain have similar performance for the Albert Barab\'{a}si random networks. This can be explained by the fact that all vertices in this network have low out-degree (1, 2 or 3 respectively) and hence trades are of small size too. Furthermore, due to the power-law in-degree distribution, many of the edges will have the same target further limiting the size of trades. However, the Global Directed Curveball algorithm again drastically improves the convergence of the perturbation score as compared to the Directed Curveball algorithm. 

\begin{figure}[!htb]
\includegraphics[width=120mm]{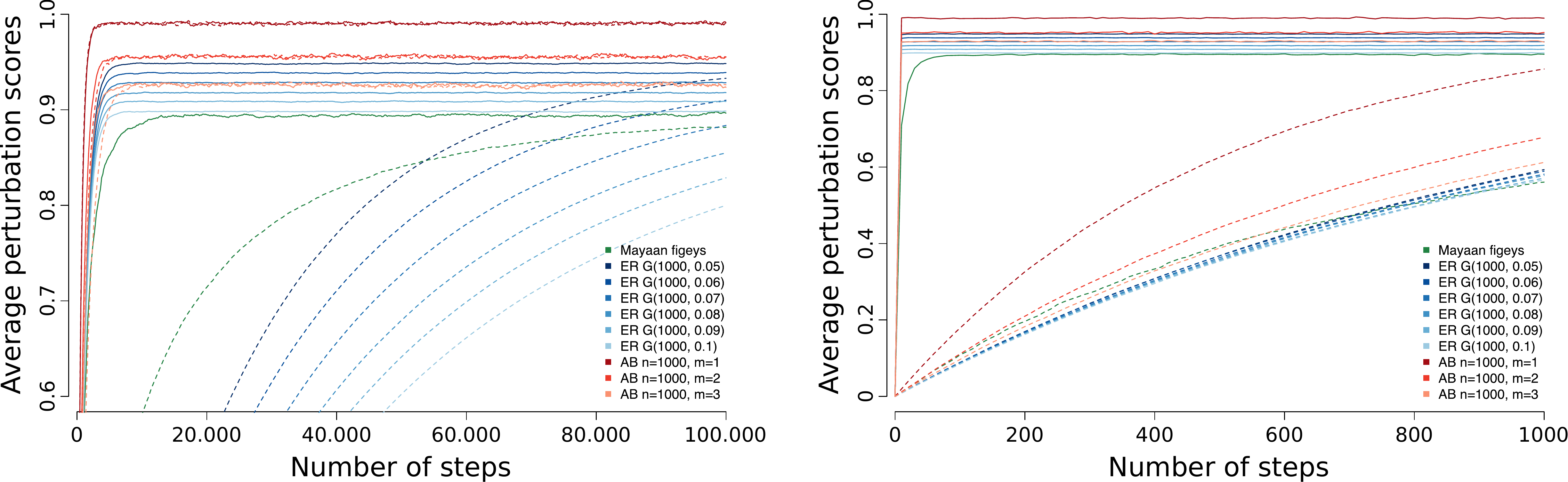}
\caption{\label{fig:digraphs}The perturbation scores of the Markov chains while randomising ten different directed graphs. On the left we compare the Directed Curveball algorithm (solid) to the Switching model (dashed), on the right we compare the Global Directed Curveball algorithm (solid) to the Directed Curveball algorithm (dashed). On the x-axis we plot the number of steps in the Markov chains, and on the y-axis the corresponding perturbation score. We run each Markov chain ten times. For the switching chain and the Directed Curveball algorithm we let $N=100.000$ steps and compute the average perturbation score over the ten runs for every 100th step. For the comparison of the Global Directed Curveball algorithm and the Directed Curveball algorithm we take just $N=$1000 steps and compute the average perturbation score over ten runs every 10th step.}
\end{figure}

Our findings for directed graphs with self-loops are identical to the findings for directed graphs and presented in the Appendix. 

Figure \ref{fig:graphs} shows our comparison of the Undirected Curveball algorithm and the switching chain for ten undirected graphs. We randomise six undirected \er networks $G(n,p)$ with 1000 vertices and varying probability $p \in \{0.05, 0.06, \dots 0.1\}$, three random networks generated using the simple preferential attachment model introduced by Albert and Barab\'{a}si with 1000 vertices and varying number of added edges $m \in \{1,2,3\}$ per step (we remove directionality from the edges), and a real graph which represents an online social network for hamster owners \cite{konect:2016:petster-friendships-hamster}. 

Our findings for graphs are very similar to our findings for directed graphs. The Directed Curveball algorithm converges much faster than the switching chain for the \er random networks and the real network. However, the algorithms have similar performance for the Albert Barab\'{a}si random networks.

\begin{figure}[!htb]
\includegraphics[width=60mm]{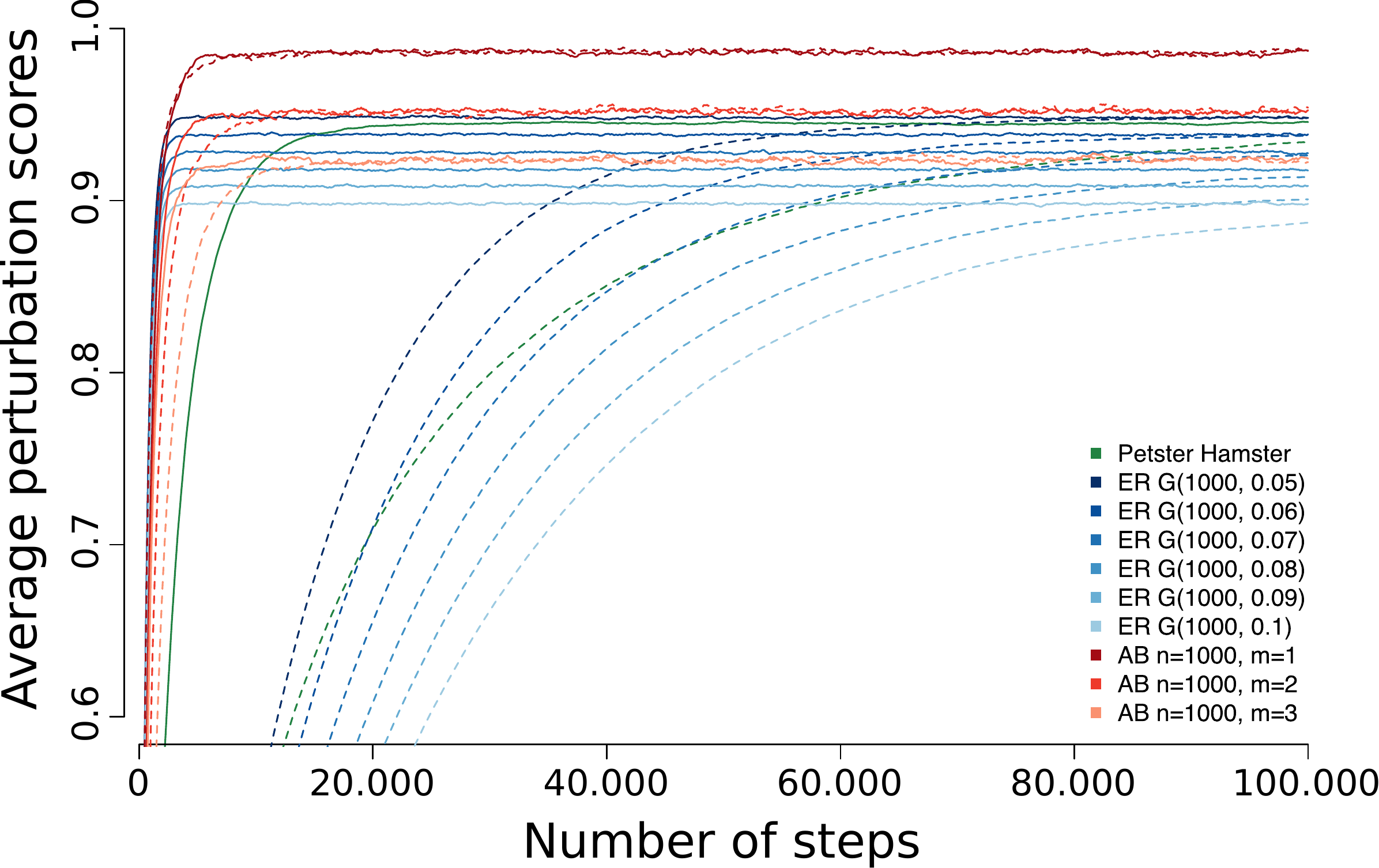}
\caption{\label{fig:graphs}The perturbation scores corresponding to the switching chain and the Undirected Curveball algorithm while randomising ten different graphs. We run each Markov chain ten times and let $N=100.000$. We compute the average perturbation score over the ten runs for every 100th step. The solid line corresponds to the Undirected Curveball algorithm and the dashed line to the switching chain.}
\end{figure}

All algorithms were implemented in the R programming language and are publicly available \cite{queenBNE_CB}.

\subsection{Theoretical questions}
Even though there is experimental evidence that the mixing time of the Curveball algorithms is much faster than that of switching models, there is currently no theoretical proof. There are few theoretical results about the rapid mixing of the switching model. For the special case of regular and semi-regular networks \cite{Kannan1994,Greenhill2011,Miklos2013}, the polynomial upper bound for the mixing time was found using a multi-commodity flow argument \cite{Jerrum1989,Sinclair1989}. These proofs rely on defining a special class of paths between all states in the state graph. Paths are chosen in such a way that the load on each edge (the number of paths it takes part in) is relatively small. The mixing time can then be bounded from above in terms of a product of these edge loads and the inverse of their transition probabilities.

Unfortunately the multi-commodity flow method can not be used to prove rapid mixing for the Curveball algorithm. The same method applied to the class of paths that was used in the switching chain cannot be used, the argument breaks down when estimating the transition probabilities. The reason for this is that the transition probabilities in the Curveball algorithm can be exponentially small with respect to the number of vertices $n$ in a network, leading to an exponential factor in the upper bound. 

We do not believe that the small transition probabilities are an \emph{actual} obstruction to fast mixing of the Curveball algorithms, since each state also has a corresponding exponential number of neighbouring states. The fastest mixing Markov chain on $N$ states has the complete graph as its state graph, with all transition probabilities equal to $\nicefrac{1}{N}$. With an exponentially large state space, these probabilities are also exponentially small. Intuitively, the Curveball algorithm is much closer to this optimal situation than the switching method. 

It appears that an altogether different method is needed to find a theoretical upper bound for the mixing time of the Curveball algorithm. This is a difficult, but important open problem. The Curveball algorithm seems to be a step in the right direction for the fast generation of random directed networks.

\section{Conclusion}
\label{sec:concl}

In this paper we introduced two extensions of the Curveball algorithm: the Directed Curveball algorithm and the Undirected Curveball algorithm. These algorithms were developed to randomise undirected and simple directed networks while fixing their degree sequence. 

It is important for random network models to sample without bias. We proved that both the Directed Curveball algorithm and the Undirected Curveball algorithm converge to the uniform distribution. Furthermore, experimental evidence shows that they do so much faster than the well-known switching models. We recommend the use of these models over that of the switching model, especially for large networks. 

We pointed out why current techniques can not be used for formal proof of rapid mixing of the Curveball algorithm. Developing new techniques and proving rapid mixing is an interesting open problem.

\begin{appendix} 

\section*{Appendix}

In the following we give an algorithm for computing a $2$-partition $P$ of a set $M=\{1,\dots,n\}$ uniformly at random. The basic idea is that there is always a pair of integers in each partition containing the minimum number of a set $M.$ The algorithm creates one pair with this number $i$ and chooses the partner $j$ randomly from $M\setminus\{i\}.$ It remains to find a $2$-partition of a set $M$ which doesn't contain $i$ and $j.$

\begin{algorithm}
    \caption{Uniform sampling of a $2$-Partition for $M=\{1,\dots,n\}$} 
		\label{AlgorithmTwoPartition}
	    \begin{algorithmic}[1]
      \REQUIRE Set $M:=\{1,\dots,n\}$.
      \ENSURE Uniform sampled $2$-partition $P=\{\{i_1,i_2\},\dots,\{i_{n-1},i_{n}\}\}$ of $M$ (for even $n$).
			        Uniform sampled $2$-partition $P=\{\{i_1\},\{i_2,i_3\},\dots,\{i_{n-1},i_{n}\}\}$ (for odd $n$).
			 \STATE Initialize $P:=\emptyset.$
			 \IF{n is odd} \STATE choose $i \in M$ at random and set $M \leftarrow M \setminus \{i\}$ and $P \leftarrow P \cup \{\{i\}\}$.			\ENDIF
      \WHILE{$|M|>2$}
				\STATE Choose the smallest number $i$ in $M$.
				\STATE Choose a random element $j$ in $M \setminus \{i\}.$
				\STATE $P \leftarrow P \cup \{\{i,j\}\}.$
				\STATE $M \leftarrow M \setminus \{i,j\}.$
			\ENDWHILE
    \end{algorithmic}
 \end{algorithm}

The while-loop in step (5) will be used at most $\nicefrac{n}{2}$ times. A careful implementation with $M$ as an initial increasing array of numbers $1,\dots,n$ requires for step (6) and deleting $i$ in (9), $O(1)$ time, for step (7) $O(1)$ time to choose $j$ \cite{Knuth:1998} and to delete it in (9). This leads to $O(n)$ time.

%\section{Proof of lemma~\ref{mylemma}} 
\TheoremInducedCycles*
\begin{proof}
We show that conditions 1.) and 2.) imply that $i,i+1,i+2$ is an induced cycle set. We prove that for any adjacency matrix $A$ corresponding to a realisation of sequence $S$, these two conditions lead to an induced cycle between vertices $i,i+1,i+2$. This shows that each possible realisation possesses such an induced cycle, and hence $i,i+1,i+2$ is an induced cycle set.

Let $A$ be any adjacency matrix corresponding to a realisation of $S$. Condition 2.) with $l=i-1$ states that the number of $1'$s in the first $i-1$ columns of $F$ and $A$ are equal. In other words, the number of $1'$s in all rows from column index 1 to $i-1$ are equal for $A$ and $F$. Observe that due to the construction of the Ferrers' matrix, the number of $1'$s in a row $j$ of $F$ from index 1 to $i-1$ must always be larger or equal to the number of $1'$s in the same row in $A$ from index 1 to $i-1$. Thus, the sequence of row sums $b^{(i-1)}_1,\dots,b^{(i-1)}_n$ for column indices $1$ to $i-1$ must be identical for matrix $A$ and $F$. (A smaller row sum in $A$ would imply another larger row sum in $A$). The same is true for the row sums $b^{(i+2)}_1,\dots,b^{(i+2)}_n$ for column indices from $1$ to $i+2$ due to condition 2.) with $l=i+2.$ 

Since $b_i=i$, $b_{i+1}=i$ and $b_{i+2}=i$ by condition 1.), we find $b^{(i-1)}_i=i-1$, $b^{(i-1)}_{i+1}=i-1$, $b^{(i-1)}_{i+2}=i-1$ and $b^{(i+2)}_i=i$, $b^{(i+2)}_{i+1}=i$, $b^{(i+2)}_{i+2}=i$. Hence the $3$x$3$-sub-matrices of $F$ and $A$ consisting of columns and rows $i,i+1,i+2$ have row sum $1$ for each row. The figure below depicts matrix $F$.

$F=$
\[
\begin{array}{llllllllllllllllll}
          &           && 1 && \ldots  && i             && i+1 && i+2             && \ldots && n   \\
					
1         & \LiKl{9} &~& \bf{0} &1& \ldots &1& \ob{1}        && \ob{1}    && \ob{1}        &&        &&   & \ReKl{9} & \rdelim\}{3}{0pt}[$\textnormal{type }(a)$]\\

\bigdots  &           && 1  &\bf{0}& \ldots &1& \mi{\myvdots} && \mi{\myvdots}       && \mi{\myvdots}                                                     \\
          &           &&  1&1& \ldots &\bf{0}&   \un{1}        && \un{1}       && \un{1}                                                            \\
i				  &           && 1  &1& \ldots&1&\ob{\bf{0}}        &&   \ob{1}     && \ob{0}                                                            &&        &&   & &b_i\\
i+1       &           && 1  &1& \ldots &1&   \mi{1} && \mi{\bf{0}} && \mi{0} &&        &&   & &b_{i+1}
                                                    \\
i+2      &           &&  1&1&\ldots  &1&    \un{1}        &&   \mi{0}     && \un{\bf{0}}&&        &&   & &b_{i+2}
\\
\bigdots  &           && ? &?&\ldots  &?&    \ob{0}        &&   \ob{0}     && \ob{0}   &&        &&   &           & \rdelim\}{3}{0pt}[$\textnormal{type }(b)$]\\                                                         
          &           && ? &?&\ldots  &?&     \mi{\myvdots} &&  \mi{\myvdots}      && \mi{\myvdots}                                                     \\
n         &           && ? &?&\ldots &?&      \un{0}        &&  \un{0}      && \un{0}                                                            \\
          &           &&   &&        && f_i          && f_{i+1}      &&f_{i+2}
\end{array}
\]

Combining conditions 2.) and 1.) we find that $f_i=a_i+1=k+1$, $f_{i+1}=a_{i+1}=k$, and $f_{i+2}=a_{i+2}-1=k-1$. Notice that these conditions imply that for any row $l$ of $F$ with $l \neq i, i+1, i+2$ the values of columns $F_{li}, F_{li+1}, F_{li+2}$ have to equal $1,1,1$ (type (a)) or $0,0,0$ (type (b)). If we allowed a row with $1,1,0$ then there has to be another row $0,0,1$, or two other rows $1,0,1$ and $0,1,1$, neither of which is possible for a Ferrers matrix. The same reason forbids $1,0,0$ as row. 

$A=$\[
\begin{array}{llllllllllllllllll}
          &           && 1 && \ldots  && i             && i+1 && i+2             && \ldots && n   \\
					
1         & \LiKl{9} &~& \bf{0} &1& \ldots &1& \ob{1}        && \ob{1}    && \ob{1}        &&        &&   & \ReKl{9} & \rdelim\}{3}{0pt}[$\textnormal{type }(a)$]\\

\bigdots  &           && 1  &\bf{0}& \ldots &1& \mi{\myvdots} && \mi{\myvdots}       && \mi{\myvdots}                                                     \\
          &           &&  1&1& \ldots &\bf{0}&   \un{1}        && \un{1}       && \un{1}                                                            \\
i				  &           && 1  &1& \ldots&1&\ob{\bf{0}}        &&   \ob{?}     && \ob{?}                                                            &&        &&   & &b_i\\
i+1       &           && 1  &1& \ldots &1&   \mi{?} && \mi{\bf{0}} && \mi{?} &&        &&   & &b_{i+1}
                                                    \\
 i+2      &           &&  1&1&\ldots  &1&    \un{?}        &&   \mi{?}     && \un{\bf{0}}&&        &&   & &b_{i+2}
\\
\bigdots  &           && ? &?&\ldots  &?&    \ob{0}        &&   \ob{0}     && \ob{0}   &&        &&   &           & \rdelim\}{3}{0pt}[$\textnormal{type }(b)$]\\                                                         
          &           && ? &?&\ldots  &?&     \mi{\myvdots} &&  \mi{\myvdots}      && \mi{\myvdots}                                                     \\
n         &           && ? &?&\ldots &?&      \un{0}        &&  \un{0}      && \un{0}                                                            \\
          &           &&   &&        && a_i          && a_{i+1}      &&a_{i+2}
\end{array}
\]

To create a realisation of column $i$ in matrix $A$ we need a $1$ less than in column $i$ of $F$ by condition 2.) with $\ell=i$. Let us assume that $A_{\ell,i}=0$ and $F_{\ell,i}=1$ with $\ell \neq i,i+1,i+2$. This is only possible when $F_\ell$ is of type (a). But then we have two different column sums $b^{(i+2)}_j$ in matrices $A$ and $F$ in contradiction to our observation above. 

Hence, we can conclude that either a) $\ell=i+1$ or b) $\ell=i+2$ ($\ell=i$ can be excluded because of the demanded diagonal entry $0$). For situation a) we find that $A_{i+1,i+2}=1$ so that $A'$ has row sum 1 for row $i+1$. Furthermore row $i$ has column sum 1 in $A'$ and hence $A_{i,i+1} = 1$ (if $A_{i,i+2} = 1$ then there has to be an index $\ell \neq i, i+1, i+2$ with $A_{l,i+2} = 0$ and $F_{l, i+2} =1$ which is again a contradiction). Similarly in situation b) we find $A_{i+2,i+1}=1$ and $A_{i,i+2}=1$. In both cases $A'$ corresponds to an induced cycle. 
\end{proof}

\ModifiedCurveballTransition*
\begin{proof}
The first part of this Lemma was already given in Section~\ref{sec_modified_curveball}. We prove the formula for $|\mathcal{P}|$ with induction on $n := |M|$. When $n=1$ and $n=2$ there is only one partition and hence $|\mathcal{P}|=1$ in both cases. When $n=3$ the $2$-partition is of the following form: $P=\{\{i_1,i_2\},\{i_3\}\}$. There are three possibilities to choose $i_3$, and $i_1$ and $i_2$ are forced by this choice. This results in $3$ possible $2$-partitions, hence $|\mathcal{P}|=3.$ When $n=4$ a $2$-partition is of the following form: $\{\{i_1,i_2\},\{i_3,i_4\}\}$. Now $i_1$ can be fixed as $i_1=1,$ because the number $1$ must be in one pair. Then there are $3$ possibilities to choose $i_2$, and after this choice $i_3$ and $i_4$ are settled. Hence, the number of $2$-partitions is $3$, i.e. $|\mathcal{P}|=3.$

Now let us assume that the claim is true for all $l \leq n-1.$ For a given $M:=\{1,\dots,n\}$ first assume $n$ is even. We can fix $i_1:=n$, because $n$ has to be in one of these pairs. For $i_2$ we have $n-1$ possible choices from $1$ to $n-1$. Let us denote this choice by $a$, i.e $i_2=a$. Now let $M_a$ equal $M\setminus \{a,n\}$. For each $|M_a|=n-2$ we can apply the induction hypothesis. Each of the $n-1$ $M_a$ can be combined with $\{n,a\}$ leading to a partition of $M$. Hence, we get for $M$, $|\mathcal{P}|=(n-1)\cdot(\Pi_{k \in \{1,3,\dots,n-3\}}n-2-k)=\Pi_{k \in \{1,3,\dots,n-1\}}n-k$. Finally, if $n$ is odd, we first need to choose an element $i_n$ randomly, and then we apply for the remaining even set $M\setminus \{i_n\}$ the formula for the even case.
\end{proof}

\begin{figure}[!htb]
\includegraphics[width=120mm]{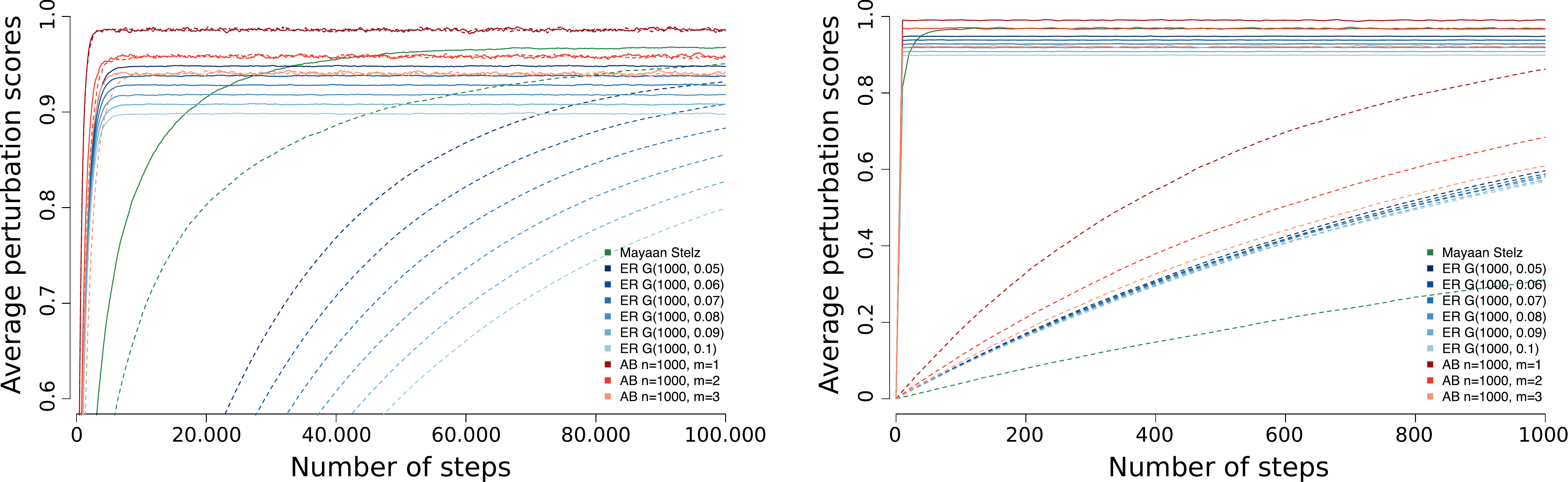}
\caption{The perturbation scores of the Markov chains while randomising ten different directed graphs with self-loops. On the left we compare the Curveball algorithm (solid) to the Switching model (dashed), on the right we compare the Global Curveball algorithm (solid) to the Curveball algorithm (dashed). On the x-axis we plot the number of steps in the Markov chains, and on the y-axis the corresponding perturbation score. We run each Markov chain ten times. For the switching chain and the Curveball algorithm we let $N=100.000$ steps and compute the average perturbation score over the ten runs for every 100th step. For the comparison of the Global Curveball algorithm and the Curveball algorithm we take just $N=$1000 steps and compute the average perturbation score over ten runs every 10th step.}
\end{figure}

\end{appendix}

\end{document}